\documentclass[11pt]{amsart}
\usepackage{amsmath,amsfonts,amssymb,amsthm}
\usepackage{graphics}
\usepackage{graphicx}
\usepackage{overpic}
\usepackage{epsfig}
\usepackage{color}
\usepackage{url}
\usepackage{epic}

\newtheoremstyle{thm}
  {9pt}{9pt}{\itshape}{}{\bfseries}{}{.5em}{}

\makeatletter
\newcommand{\longdash}[1][2em]{%
  \makebox[#1]{$\m@th\smash-\mkern-7mu\cleaders\hbox{$\mkern-2mu\smash-\mkern-2mu$}\hfill\mkern-7mu\smash-$}}
\makeatother
\newcommand{\omitskip}{\kern-\arraycolsep}

\theoremstyle{thm}
\newtheorem{thm}{Theorem}[section]
\newtheorem{cor}[thm]{Corollary}
\newtheorem{lemma}[thm]{Lemma}

\newtheoremstyle{defin}
  {9pt}{9pt}{}{}{\bfseries}{}{.5em}{}
\theoremstyle{defin}

\newtheorem{cl}{Claim}
\newtheoremstyle{exm}
  {9pt}{9pt}{}{}{\scshape}{}{.5em}{}
\theoremstyle{exm}

\newtheoremstyle{proof}
  {}{}{}{}{\itshape}{:}{.5em}{}
\theoremstyle{proof}

\newcommand{\Z}{{\mathbb Z}}
\newcommand{\R}{{\mathbb R}}

\DeclareMathOperator{\conv}{Conv}

\def\sm{\smallsetminus}

\def\<{\langle}
\def\>{\rangle}

\def\0{{\mathbf 0}}

\def\.{\hskip.06cm}

\def\conv{{\text {\rm {conv}} }}

\newcommand{\V}{\mathcal{V}}

\renewcommand{\mod}{\text{mod} \,}

\title{How many $n$-vertex triangulations does the $3$-sphere have?}
\author{Eran Nevo and Stedman Wilson}
\address{
Department of Mathematics,
Ben Gurion University of the Negev,
Be'er Sheva 84105, Israel
}
\email{nevoe@math.bgu.ac.il}
\email{stedmanw@gmail.com}

\thanks{Research was partially supported by Marie Curie grant IRG-270923 and ISF grant.}

\begin{document}

\maketitle

\begin{abstract}
It is known that the $3$-sphere has at most $2^{O(n^2 \log n)}$ combinatorially distinct triangulations with $n$ vertices. Here we construct at least $2^{\Omega(n^2)}$ such triangulations.
\end{abstract}

\section{Introduction}
For $d\geq 3$ fixed and $n$ large,
Kalai \cite{Kalai-manyspheres} constructed $2^{\Omega(n^{\lfloor d/2 \rfloor})}$ combinatorially distinct
 $n$-vertex triangulations of the $d$-sphere (the \emph{squeezed} spheres), and concluded from
Stanley's upper bound theorem for simplicial spheres~\cite{Stanley:CohenMacaulayUBC-75}
an upper bound of $2^{O(n^{\lceil d/2 \rceil}\log n)}$ for the number of such triangulations.  In fact, this upper bound readily follows from the Dehn-Sommerville relations, as they imply that the number of $d$-dimensional faces is a linear combination of the number of faces of dimension $\leq \lceil d/2\rceil-1$, and hence is at most $O(n^{\lceil d/2\rceil})$. Thus, as already argued in \cite{Kalai-manyspheres}, the number of different triangulations is at most
$\binom{\binom{n}{d}}{O(n^{\lceil{d/2}\rceil})}/n!$, namely at most $2^{O(n^{\lceil d/2 \rceil}\log n)}$.

For $d$ odd this leaves a big gap between the upper and lower bounds; most striking for $d=3$.
Pfeifle and Ziegler \cite{Pfeifle-Ziegler:many3-spheres} constructed $2^{\Omega(n^{5/4})}$ combinatorially different $n$-vertex triangulations of the $3$-sphere. Combined with the $2^{O(n \log n)}$ upper bound for the number of combinatorial types of $n$-vertex simplicial $4$-polytopes \cite{GoodmanPollack:fewPolytopes-86} (see also \cite{Alon:fewPolytopes}) it shows that most triangulations of the $3$-sphere are not combinatorially isomorphic to boundary complexes of simplicial polytopes.

The bound in \cite{Pfeifle-Ziegler:many3-spheres} is obtained by constructing a polyhedral $3$-sphere with $\Omega(n^{5/4})$ combinatorial octahedra among its facets.
Our bound will follow from constructing a polyhedral $3$-sphere with $\Omega(n^2)$ combinatorial \emph{bipyramids} among its facets. (A bipyramid is the unique simplicial $3$-polytope with $5$ vertices.)
The idea for the construction is as follows: consider the boundary complex $C$ of the cyclic $4$-polytope with $n$ vertices.
\begin{itemize}
\item{}Find particular $\Theta(n)$ simplicial $3$-balls contained in $C$, with disjoint interiors.
\item{}On the boundary of each such $3$-ball find particular $\Theta(n)$ pairs of adjacent triangles  (each pair forms a square), such that these squares have disjoint interiors, and the missing edge in each such pair is an interior edge of its $3$-ball.
\item{}Replace the interior of each such $3$-ball with the cone from a new vertex over each boundary square (forming a bipyramid) and over each remaining boundary triangle (forming a tetrahedron).
\item{}Show that the particular $3$-balls and squares chosen have the property that the above construction results in a polyhedral $3$-sphere.
\end{itemize}

Specifically, we prove the following.

\begin{thm}
For each $n \geq 1$, there exists a $3$-dimensional polyhedral sphere with $5n+4$ vertices, such that $n^2$ of its facets are combinatorially equivalent to a bipyramid.
\label{t:complex}
\end{thm}

Erickson \cite[Sec.8]{Erickson} asked whether there exist $4$-polytopes on $n$ vertices with $\Omega(n^2)$ non-simplicial facets, conjectured there are none, and further conjectured that there are no such polyhedral $3$-spheres.
The latter is refuted by Theorem \ref{t:complex}.
We leave open the question of whether the polyhedral $3$-spheres constructed in Theorem \ref{t:complex} are combinatorially equivalent to the boundary complexes of some $4$-polytopes.

Note that each of the bipyramids in the above theorem can be triangulated independently in two ways --- either into $2$ tetrahedra by inserting its missing triangle or into $3$ tetrahedra by inserting its missing edge --- to obtain a triangulation of the $3$-sphere.
Thus, for $v$ the number of vertices and $m$ the number of bipyramids, we obtain at least $\frac{2^m}{v!}$ combinatorially distinct triangulations.
This gives the following result.

\begin{cor}
The $3$-sphere admits $2^{\Omega(n^2)}$ combinatorially distinct triangulations on $n$ vertices. $\square$
\label{t:sphere}
\end{cor}

\section{Preliminaries}
For background on polytopes, the reader can consult~\cite{Ziegler} or~\cite{Grunbaum}, and for background on simplicial complexes
see e.g.~\cite{Munkres}.

Let $X$ denote a complex, simplicial or polyhedral.
In what follows, always assume $X$ to be \emph{pure}, namely all its maximal faces with respect to inclusion have the same dimension.  We write $\V(X)$ for the set of all vertices of $X$.  By a \emph{facet} of $X$ we mean a face $F \in X$ of maximal dimension.  We use the term $k$-face as a shorthand for $k$-dimensional face, as usual.  Similarly, we call a complex $X$ a \emph{$k$-complex} if all of its facets are $k$-faces.

For a polyhedral ball $X$ we denote by $\partial X$ the \emph{boundary complex} of $X$.  That is, $\partial X$ is the subcomplex of $X$ whose facets are precisely the faces of $X$ that are contained in exactly one facet of $X$.  In particular, if $X$ is a $k$-complex homeomorphic to the  $k$-ball, then $\partial X$ is $(k - 1)$-dimensional, homeomorphic to the $(k-1)$-sphere.  We say that a face $F \in X$ is \emph{interior} to a polyhedral ball $X$ if $F \notin \partial X$.
For $X$ a simplicial complex, the \emph{link} of a face $F\in X$ is the subcomplex
$\{T\in X:\ F\cap T=\emptyset, \ F\cup T\in X\}$, and its (closed) \emph{star} is the subcomplex generated by the faces
$\{T\in X:\ F\subseteq T\}$ under taking subsets.
If $P$ is a \emph{simplex}, then we will identify $P$ with its set of vertices $\V(P)$, or with the simplicial complex $2^{\V(P)}$, when convenient.

We will make use of the following arithmetic notation.  For any integer $n \geq 1$, we will use $[n]$ as a shorthand for $[1, n] \cap \Z$, the set of all integers from $1$ to $n$.  For a real number~$r$, we will denote by $\lfloor r \rfloor$ the floor of $r$, and by $\lceil r \rceil$ the ceiling of $r$.  We will also use the notation $\sigma \in \{+, -\}$, with the convention that $-\sigma \in \{+, -\}$ is the sign different from $\sigma$.

The foundation for our construction is the \emph{cyclic polytope}, so we restate its definition here.  The \emph{moment curve} in $\R^d$ is the curve $\alpha_d : \R \rightarrow \R^d$ defined by
\[
\alpha_d(t) =  (t, t^2, t^3, \dots, t^d).
\]

The convex hull of the image of $[n]$ under $\alpha_d$, which we will denote by $C(n, d)$, is the \emph{cyclic $d$-polytope} with the $n$ vertices $\alpha_d(1), \alpha_d(2), \ldots, \alpha_d(n)$.  The faces of the cyclic polytope admit a simple combinatorial description, called \emph{Gale's evenness condition} (see e.g.~\cite{Grunbaum}, p.~62, and~\cite{Ziegler}, p.~14).  We restate this property here as a lemma.

\begin{lemma} All facets of $C(n, d)$ are $(d - 1)$-simplices.  Furthermore, for any set of $d$ integers $I \subset [n]$, the convex hull $\conv(\alpha_d(I))$ is a facet of $C(n, d)$ if and only if for every $x, y \in [n] \sm I$, there are an even number of elements $z \in I$ satisfying $x < z < y$.
\label{l:cyclic}
\end{lemma}


\section{Construction of the polyhedral sphere}

Consider the cyclic $4$-polytope \[C(4n+4, 4) = \conv(\alpha_4([4n+4])).\]  Let $P(n)$ be a polyhedral complex that is combinatorially isomorphic to the boundary complex of $C(4n+4, 4)$.  We label the set of vertices of $P(n)$ by $[4n+4]$,  ordered so that each vertex $i$ of $P(n)$ corresponds to the vertex $\alpha_4(i)$ of $C(4n+4, 4)$, under this isomorphism.  Note that $P(n)$ is homeomorphic to the $3$-sphere.  By Lemma~\ref{l:cyclic}, all the facets of $P(n)$ are tetrahedra.  That is, $P(n)$ is a simplicial complex.

In what follows, we will describe certain faces and subcomplexes of $P(n)$, which we will ultimately use to construct the polyhedral $3$-sphere of Theorem~\ref{t:complex}.




We define a set of integers
\[A(n) = \{m \in [n + 2, \ 3n + 1] \ | \ m = 2k, \ k \in \Z \}.\]
Therefore $|A(n)|=n$.
For $a \in A(n)$ and $u \in [n]$ we define the collections of vertices
\[
	\begin{array}{llllll}
		I(a, u, 1) & = & \{a - u - 1, & a - u, & a + u, & a + u + 1\}, \\
		I(a, u, 2) & = & \{a - u - 1, & a - u, & a + u + 1, & a + u + 2\}, \\
		I(a, u, 3) & = & \{a - u, & a - u + 1, & a + u + 1, & a + u + 2\}.
	\end{array}
\]
Then every set $I(a, u, i)$ is a facet of $P(n)$.  This is because $I(a, u, i)$ satisfies the criteria of Lemma~\ref{l:cyclic} in the case $d = 4$, hence $\alpha_4(I(a, u, i))$ is the set of vertices of a facet of $C(4n+4, 4)$.
For later use, let $I_-(a, u, i)$ (resp. $I_+(a, u, i)$) denote the smallest (resp. largest) two elements in $I(a,u,i)$.




We need the following auxiliary lemma.
\begin{lemma}
For all $a\in A(n)$, $u,u'\in [n]$ and $i, j \in [3]$, if $u' \leq u - 1$, then
\[I(a, u, i) \cap I(a,u', j) \subseteq
	\begin{cases}
		\{a - u, \ a + u, \ a + u + 1\} & i = 1 \\
		\{a - u, \ a + u + 1\} & i = 2 \\
		\{a - u, \ a - u + 1, \ a + u + 1\} & i = 3
	\end{cases}
 \]
\label{l:vertices}
\end{lemma}

\begin{proof}
Let $m \in I(a, u', j)$.  Then \[a - u' - 1 \leq m \leq a + u' + 2.\]  Since $u' \leq u - 1$, we have $a - u \leq a - u' - 1$ and $a + u' + 2 \leq a + u + 1$.  Therefore
\[a - u \leq m \leq a + u + 1.\]
The lemma now follows immediately from the definition of the sets $I(a, u, i)$.
\end{proof}

For each fixed $a \in A(n)$, we consider the collection
\[B_0(a) = \{I(a, u, i) \ | \ u \in [n], i \in [3]\}\]
of $3n$ facets of $P(n)$.  Let $B(a)$ denote the simplicial complex obtained by taking the closure of $B_0(a)$ under subsets.
We chose the simplicial complex $B(a)$ for two main reasons, which we establish in the following lemmas.  The first is that $B(a)$ is a $3$-ball (see Lemma \ref{l:ball}).  The second is that any two such balls $B(a), B(a')$ intersect ``minimally" (see Lemma~\ref{l:intface} and Lemma~\ref{l:intboundary}).  These two facts will be crucial to our construction of the polyhedral $3$-sphere of Theorem~\ref{t:complex}.

\begin{lemma}  For each $a \in A(n)$, the simplicial complex $B(a)$ is a shellable simplicial $3$-ball.
\label{l:ball}
\end{lemma}

Before giving a formal proof, let us describe a way to ``visualize" $B(a)$. For fixed $i=1,2,3$ let $L(i)$ denote the ``chain" of tetrahedra $L(i)=\{I(a,u,i):\ u\in [n]\}$. In $L(i)$ a tetrahedron $I(a,u,i)$ intersects only the tetrahedra right before and after it in the chain; specifically, it intersects $I(a,u+1,i)$ in one edge and $I(a,u-1,i)$ in its opposite edge. From this description it is easy to see that putting $L(2)$ ``on top of" $L(1)$ forms a simplicial $3$-ball; see Figure~\ref{f:layers}. However, this ball has no interior edges, a property needed later (see Lemma \ref{l:edge}). To fix this, we put $L(3)$ on top of $L(2)$, which gives the simplicial ball $B(a)$.

\begin{figure}[t!]
 \begin{center}
  \begin{overpic}[width=15cm, height=3.7cm]{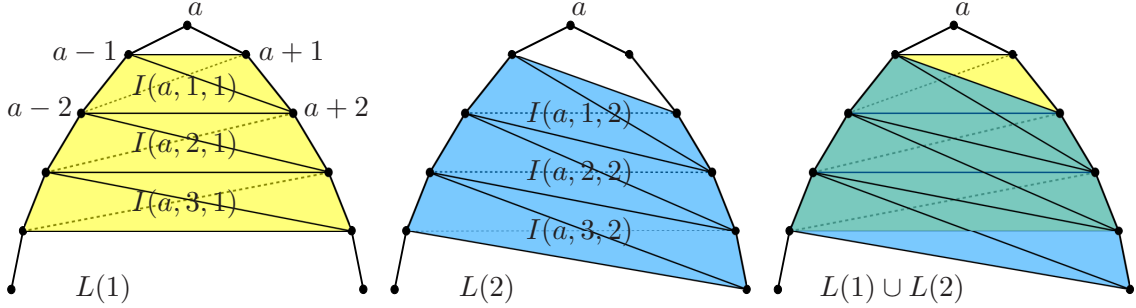}
	 \put(16, 25){$a$}
	 \put(50, 25){$a$}
	 \put(84, 25){$a$}
	 \put(4, 21){$a - 1$}
	 \put(22.3, 21){$a + 1$}
	 \put(0, 16){$a - 2$}
	 \put(26.3, 16){$a + 2$}
	 \put(11, 17.9){$I(a, 1, 1)$}
	 \put(11, 12.8){$I(a, 2, 1)$}
	 \put(11, 7.6){$I(a, 3, 1)$}
	 \put(46, 15.5){$I(a, 1, 2)$}
	 \put(46, 10.3){$I(a, 2, 2)$}
	 \put(46, 5.1){$I(a, 3, 2)$}
	 \put(6, 0){$L(1)$}
	 \put(40, 0){$L(2)$}
	 \put(72, 0){$L(1) \cup L(2)$}
  \end{overpic}
  \caption{Two of the three layers of the ball $B(a)$, together with their union.  For clarity, we display only $3$ tetrahedra in each layer $L(i)$ for $i=1,2$, namely those $I(a, u, i)$ with $u = 1, 2, 3$.}
  \label{f:layers}
 \end{center}
\end{figure}

\begin{proof}[Proof of Lemma \ref{l:ball}]
We exhibit a shelling order of the facets of $B(a)$. In particular, we define an ordering $F_1, F_2, \ldots, F_{3n}$ of the facets of $B(a)$, such that, if $G_k$ denotes the closure under subsets of $\{F_1, F_2, \ldots, F_k\}$, then the simplicial complex $G_k$ is a $3$-ball for all $k \in [3n]$.

This order of the facets is easy to describe.  For $i \in \Z$, let $r(i)$ denote the unique element of $\{1, 2, 3\}$ for which $i \equiv r(i) \ (\mod 3)$.  We also write $u_i = \lceil \frac{i}{3} \rceil$.  Then we define
\[F_i = I(a, u_i, r(i)), \quad \quad \quad i \in [3n].\]

We check that this is indeed a shelling order for $B(a)$.  The facet $F_1$ is a tetrahedron, so $G_1$ forms a $3$-ball.  Now assume inductively that $G_k$ forms a $3$-ball, for some $k \in [3n - 1]$.

Note first of all that since $P(n)$ is a $3$-sphere, every triangle $T \in P(n)$ is a face of exactly two tetrahedra of $P(n)$.  Therefore, for any triangle $T \in F_{k + 1} \cap G_k$, we must have
\begin{equation}
T \in \partial G_k
\label{e:shellable}
\end{equation}
as otherwise, since $G_k$ is a $3$-ball, the triangle $T$ would be a face of three tetrahedra in $P(n)$, namely two tetrahedra of $G_k$, and the tetrahedron $F_{k + 1}$.

We now show that $G_{k + 1}$ is a $3$-ball.  We consider separately the cases $r(k) = 1, 2, 3$.  When $r(k) = 1$, we have $u_k = \lceil \frac{k}{3} \rceil = \lceil \frac{k + 1}{3} \rceil = u_{k + 1}$.  Therefore $F_k = I(a, u_k, 1)$ and $F_{k + 1} = I(a, u_k, 2)$.
Thus $F_{k + 1} \cap F_k$ is the triangle
\[T_{k, 1} = \{a - u_k - 1, \ a - u_k, \ a + u_k + 1\}.\]
If $i \leq k - 1$ then $u_i \leq u_k - 1$, so by Lemma \ref{l:vertices},
\[F_{k + 1} \cap F_i = I(a, u_i, r(i)) \cap I(a, u_k, 2) \subseteq  \{a - u_k, \ a + u_k + 1\} \subset T_{k, 1}. \]
So we have $F_{k + 1} \cap F_i \subseteq T_{k, 1}$ for all $i \leq k$, and $F_{k + 1} \cap F_k = T_{k, 1}$.  Thus $F_{k + 1} \cap G_k = T_{k, 1}$.  By the inductive hypothesis $G_k$ is a $3$-ball.  Therefore $G_{k + 1}$ is the union of the two $3$-balls $G_k$ and $F_{k + 1}$, the intersection of which is the $2$-ball $T_{k, 1}$.  By~(\ref{e:shellable}), this $2$-ball is contained in the boundary of both $G_k$ and $F_{k + 1}$. Hence $G_{k + 1}$ is a $3$-ball in the case $r(k) = 1$.

When $r(k) = 2$, we have $u_{k - 1} = u_k = u_{k + 1}$, so $F_{k - 1} = I(a, u_k, 1)$, $F_k = I(a, u_k, 2)$, and $F_{k + 1} = I(a, u_k, 3)$.  Thus $F_{k + 1} \cap F_k$ is the triangle
\[T_{k, 2} = \{a - u_k, \ a + u_k + 1, \ a + u_k + 2\},\]
and
\[F_{k + 1} \cap F_{k - 1} = \{a - u_k, \ a + u_k + 1\} \subset T_{k, 2},\]
finishing the proof if $k=2$.
Furthermore, if $k>2$, then $F_{k - 3} = I(a, u_k - 1, 2)$, so $F_{k + 1} \cap F_{k - 3}$ is the triangle
\[T_{k, 2}' = \{a - u_k, \ a - u_k + 1, \ a + u_k + 1\}.\]
If $i \leq k - 2$ then $u_i \leq u_k - 1$, so by Lemma \ref{l:vertices},
\[F_{k + 1} \cap F_i = I(a, u_i, r(i)) \cap I(a, u_k, 3) \subseteq \{a - u_k, \ a - u_k + 1, \ a + u_k + 1\} = T_{k, 2}'.\]
Therefore $F_{k + 1} \cap F_i$ is contained in the complex formed by $T_{k, 2}$ and $T_{k, 2}'$ for all $i \leq k$, and $F_{k + 1} \cap F_k = T_{k, 2}$ and $F_{k + 1} \cap F_{k - 3} = T_{k, 2}'$.  We conclude that $F_{k + 1} \cap G_k$ has the two facets $T_{k, 2}$ and $T_{k, 2}'$, which form a $2$-ball. Together with~(\ref{e:shellable}) it follows that $G_{k + 1}$ is a $3$-ball in the case $r(k) = 2$.

Finally, when $r(k) = 3$, we have $r(k + 1) = 1$, and $u_{k + 1} = u_k + 1$ and $u_{k - 1} = u_k$.  Therefore $F_{k + 1} = I(a, u_k + 1, 1)$ and $F_{k - 1} = I(a, u_k, 2)$, so $F_{k + 1} \cap F_{k - 1}$ is the triangle
\[T_{k, 3} = \{a - u_k - 1, \ a + u_k + 1, \ a + u_k + 2\}.\]
Also, if $i \leq k$, then $u_i \leq u_{k + 1} - 1$, so by Lemma \ref{l:vertices},
\[F_{k + 1} \cap F_i = I(a, u_i, r(i)) \cap I(a, u_k + 1, 1) \subseteq \{a - u_k - 1, \ a + u_k + 1, \ a + u_k + 2\} = T_{k, 3}.\]
Therefore, $F_{k + 1} \cap F_i \subseteq T_{k, 3}$ for all $i \leq k$, and $F_{k + 1} \cap F_{k - 1} = T_{k, 3}$.  We conclude that $F_{k + 1} \cap G_k = T_{k, 3}$.  This and~(\ref{e:shellable}) imply that that $G_{k + 1}$ is a $3$-ball in the case $r(k) = 3$.
\end{proof}

We now show that the intersection of two different balls $B(a)$ and $B(a')$ does not contain a triangle.  As an immediate consequence, the balls $B(a)$ and $B(a')$ intersect only in their boundaries, which we state as a separate lemma.

\begin{lemma}
For distinct $a, a' \in A(n)$, the intersection $B(a) \cap B(a')$ does not contain a $2$-face of $P(n)$.
\label{l:intface}
\end{lemma}

\begin{proof}
Let $a, a' \in A(n)$, with $a' < a$.  Suppose there is a triangle $T \subset B(a) \cap B(a')$. As $T$ belongs to some tetrahedron $I(a,u,i)$ with $u\in [n]$ and $i\in [3]$, there is a unique way to write $T= \{k_1, k_2, k_3\}$ such that $|k_1-k_2|=1$, and then it satisfies $\nobreak{4a-1\leq k_1+k_2+2k_3\leq 4a+5}$.  As $T$ also belongs to some tetrahedron $I(a',u',i')$ with $u'\in [n]$ and $i'\in [3]$, it follows that $\nobreak{4a'-1\leq k_1+k_2+2k_3\leq 4a'+5}$. However, by definition of $A(n)$, we have $a'\leq a-2$, hence the intervals $[4a'-1,4a'+5]$ and $[4a-1,4a+5]$ are disjoint, a contradiction.
\end{proof}

\begin{lemma}
For distinct $a, a' \in A(n)$, we have $B(a) \cap B(a') \subset \partial B(a) \cap \partial B(a')$.
\label{l:intboundary}
\end{lemma}

\begin{proof}
Let $a, a' \in A(n)$, with $a' \neq a$.  Let $F \in B(a) \cap B(a')$ and suppose by contradiction that $F$ is interior to one of the $3$-balls, say $B(a)$.  Since $F \in B(a')$, and $P(n)$ is a $3$-sphere, this implies that the closed star of $F$ in $B(a')$ is a subcomplex of $B(a)$.  In particular, $B(a) \cap B(a')$ contains a tetrahedron, hence it contains a triangle.  This contradicts Lemma~\ref{l:intface}.
\end{proof}

To understand the boundary complex $\partial B(a)$ of each ball $B(a)$, we introduce the following notation.
For $a \in A(n)$, $u \in [n]$ and $i \in [3]$, let
\[
	\begin{array}{llrllll}
		x_{-}(a, u, 1) & = & a -  u - 1, & \quad & x_{+}(a, u, 1) & = & a + u, \\
		x_{-}(a, u, 2) & = & a - u - 1, & \quad & x_{+}(a, u, 2) & = & a + u + 2, \\
		x_{-}(a, u, 3) & = & a - u + 1, & \quad & x_{+}(a, u, 3) & = & a + u + 2.
	\end{array}
\]
Our next result characterizes the boundary complex of each ball $B(a)$.

\begin{lemma}
For every $a \in A(n)$, the $2$-faces of the boundary complex $\partial B(a)$ are exactly the triangles
\[
		I_\sigma(a, u, i) \cup \{x_{-\sigma}(a, u, i)\}, \\
\]
for $u \in [n]$, $i \in [3]$, and $\sigma \in \{-, +\}$.
\label{l:boundary}
\end{lemma}

\begin{proof}
Let $T \in B(a)$ be a triangle, so $T$ is in $I(a,u,i)$ for some $u\in [n]$ and $i\in [3]$, and can be written uniquely as
$T = \{k_1, k_2, k_3\}$, where $|k_1 - k_2|=1$. Together with $T$, one of $k_3\pm 1$ forms the tetrahedron $I(a,u,i)$, and the other forms a tetrahedron $I'\in P(n)$. Then $T\in \partial B(a)$ if and only if $I'\neq I(a,u',i')$ for all $u'\in [n]$ and $i'\in [3]$. Essentially, fixing~$a$ and~$u$, there are $12$ triangles for which we need to check this (though some work can be saved). We exhibit here a sample of these computations.
A key invariant to compute is the \emph{label average} $e(F)$ of the vertices of a tetrahedron $F$.
For $I(a,u,i)$ this equals
\[e(I(a, u, i)) = \frac{1}{4}\sum_{k \in I(a, u, i)}k = a+\frac{i-1}{2}.\]
In particular, $e(I(a, u, i)) = e(a, i)$ does not depend on $u$.

Let $T=I(a,u,1)\setminus \{a+u+1\}$, so $k_3=a+u$. For $I'=T\cup\{k_3 -1\}$, we obtain $e(I')=a-1/2 <a$, thus $e(I')$ is smaller then all $e(I(a,u',i'))$, implying that $T\in \partial B(a)$.

Let $T=I(a,u,1)\setminus \{a+u\}$, so $k_3=a+u+1$. For $I'=T\cup\{k_3 +1\}$, we obtain $I'=I(a,u,2)\in B(a)$, implying that $T\notin \partial B(a)$.

We leave the other $10$ checks to the reader.
\end{proof}

The classification of Lemma~\ref{l:boundary} yields an important fact about the edges of $B(a)$, as follows. Define the edge $E(a, u)$ of $P(n)$ by
\[
	E(a, u) = \{a - u, a + u + 1\}.
\]
Clearly $E(a, u)$ is an edge of $I(a, u, 1)$, hence an edge of $B(a)$.  As it turns out, the interior edges of $B(a)$ are exactly the edges $E(a, u)$, a consequence of Lemma~\ref{l:boundary}.  This is the content of the next lemma.

\begin{lemma}
The interior edges of $B(a)$ are exactly the edges $\{E(a, u):\ u \in [n]\}$.
\label{l:edge}
\end{lemma}

\begin{proof}
Fix $u\in [n]$.  By Lemma~\ref{l:boundary} we see that, among the edges of the $3$ tetrahedra $I(a,u,i)$, with $i\in [3]$, the only one not belonging to a boundary triangle of $B(a)$ is $E(a,u)$. Thus, all interior edges of $B(a)$ must be of the form $E(a,u)$ for some $u \in [n]$. To verify that each edge $E(a,u)$ is indeed interior in $B(a)$, we check that the link of $E(a,u)$ in $B(a)$ is a cycle. Indeed, this link contains the $4$-cycle $(a-u-1,a+u,a-u+1,a+u+2)$, and hence is equal to it (as the link in the entire complex $P(n)$ is a cycle).
\end{proof}

We now consider the two triangles of $I(a, u, 1)$ not having $E(a, u)$ as an edge, namely
\[
		T_\sigma(a, u) = I_\sigma(a, u, 1) \cup \{x_{-\sigma}(a, u, 1)\}
\]
for $\sigma \in \{-, +\}$.  Explicitly,
\[
	\begin{array}{lll}
		T_-(a, u) & = &  \{a - u - 1, \ a - u, \ a + u\}, \\
		T_+(a, u) & = & \{a - u - 1, \ a + u, \ a + u + 1\}.
	\end{array}
\]
We define $D(a, u)$ to be the $2$-dimensional simplicial complex obtained as the closure of $\{T_-(a, u), T_+(a, u)\}$ under subsets.  From Lemma~\ref{l:boundary}, we see that each $T_\sigma(a, u)$ is a boundary triangle of $B(a)$.  That is, $D(a, u)$ is a subcomplex of $\partial B(a)$.  We define $R(a, u)$ to be the intersection of the two triangles of $D(a, u)$.   That is,
\[
	R(a, u) = T_-(a, u) \cap T_+(a, u) = \{a - u - 1, a + u\}.
\]
Since $R(a, u)$ is an edge, it follows that $D(a, u)$ is a $2$-ball.  Note also that $R(a, u)$ is the unique interior edge of $D(a, u)$.

Understanding the intersection of distinct disks $D(a, u)$ and $D(a', u')$ is crucial for constructing the polyhedral $3$-sphere of Theorem~\ref{t:complex}.  The relevant properties of this intersection are stated in the following lemma.

\begin{lemma}
For $(a,u)\neq (a',u')$, the disks $D(a,u)$ and $D(a',u')$ intersect in a single face.  When $a = a'$, this intersection lies on the boundary of both disks.
\label{l:thebigone}
\end{lemma}

\begin{proof}
Let $a, a' \in A(n)$ and $u, u' \in [n]$ such that $(a, u) \neq (a', u')$.

First we consider the case $a = a'$ and $u' < u$.
Then $u' \leq u - 1$, so by Lemma~\ref{l:vertices}, we have
\[
T_{\sigma'}(a, u') \cap T_{\sigma}(a, u) \subset I(a, u', 1) \cap I(a, u, 1) \subseteq \{a - u, a + u, a + u + 1\},
\]
for every choice of $\sigma, \sigma' \in \{-, +\}$.  Thus
\begin{equation}
\V(D(a, u) \cap D(a, u')) \subseteq \{a - u, a + u, a + u + 1\}.
\label{e:sameball}
\end{equation}
By definition,
\[
	\begin{array}{lll}
		T_-(a, u') & = &  \{a - u' - 1, \ a - u', \ a + u'\}, \\
		T_+(a, u') & = & \{a - u' - 1, \ a + u', \ a + u' + 1\}.
	\end{array}
\]
Since $u' \leq u - 1$, it follows that
\[
	\begin{array}{lll}
		T_-(a, u') \cap \{a - u, a + u, a + u + 1\} & \subseteq &  \{a - u\}, \\
		T_+(a, u') \cap \{a - u, a + u, a + u + 1\} & \subseteq & \{a - u, a + u\}.
	\end{array}
\]
This and~(\ref{e:sameball}) imply that
\[
\V(D(a, u) \cap D(a, u')) \subseteq \{a - u, a + u\}.
\]
As $\{a-u,a+u\}$ is an edge of $T_-(a,u)$ we conclude that
$D(a, u) \cap D(a, u')$ is either an edge or a vertex or empty.  Finally, observe that $\{a - u, a + u\}$ is not one of the edges $R(a, u)$, $R(a, u')$.  Therefore
\[D(a, u) \cap D(a, u') \subset \partial D(a, u) \cap \partial D(a, u').\]

Now consider the case $a' < a$.  In this case, $E(a, u) \neq E(a', u')$.  For otherwise, we may take the average of the vertices of $E(a, u) = E(a', u')$, to obtain
\[
a + \frac{1}{2} = \ \frac{1}{2}\sum_{k \in E(a, u)}k \ = a' + \frac{1}{2},
\]
from which we conclude $a = a'$.

From the definitions, we see that $E(a, u)$ is the only missing edge of $D(a, u)$ (namely, a non-edge of $D(a,u)$ whose vertices are in $D(a,u)$), and $E(a', u')$ is the only missing edge of $D(a', u')$.  It follows that $D(a, u) \cap D(a', u')$ does not contain both vertices of $E(a, u)$, as if it did, the fact that $E(a, u) \neq E(a', u')$ implies that $E(a, u)$ is an edge of $D(a', u')$.  But $E(a,u)$ is interior to $B(a)$, contradicting Lemma \ref{l:intboundary}.
Similarly, $\nobreak{D(a, u) \cap D(a', u')}$ does not contain both vertices of $E(a', u')$.  That is,
\begin{equation}
E(a, u), E(a', u') \not \subset \V(D(a, u) \cap D(a', u')).
\label{e:edgesubset}
\end{equation}

Now suppose (for a contradiction) that $D(a, u) \cap D(a', u')$ has at least $3$ vertices.  Then~(\ref{e:edgesubset}) implies that $D(a, u) \cap D(a', u')$ must contain one of the triangles $T_\sigma(a, u)=T_{\sigma'}(a', u')$, contradicting Lemma~\ref{l:intface}.  It follows that $D(a, u) \cap D(a', u')$ is  either an edge or a vertex or empty.
\end{proof}

We are now ready to construct the polyhedral $3$-sphere that proves Theorem \ref{t:complex}.

\begin{proof}[Proof of Theorem \ref{t:complex}]
For each $a \in A(n)$, we do the following to $P(n)$.  Remove all interior faces of $B(a)$.  Add a new vertex $q(a)$ to $P(n)$, and cone from this vertex to $\partial B(a)$, to obtain a new simplicial $3$-ball, call it $B'(a)$.

By Lemma \ref{l:intboundary}, the collection of all tetrahedra that results from the above process (both new tetrahedra and tetrahedra which were not removed from $P(n)$) is a simplicial complex, call it $P'(n)$.  Since each $B'(a)$ is a $3$-ball with the same boundary as $B(a)$, the complexes $P'(n)$ and $P(n)$ are homeomorphic.  That is, $P'(n)$ is also a $3$-sphere.

Now, for each $a \in A(n)$ and $u \in [n]$, remove from $P'(n)$ the triangle $T'(a, u)$ defined by
\[T'(a, u) = \{q(a)\} \cup R(a, u),\]
merging the two tetrahedra it borders into the same facet.
Call the resulting collection of faces $Q(n)$.  Recall that the edge $R(a, u)$ of $T'(a, u)$ is the unique interior edge of the disk $D(a, u)$.  Therefore, it follows from Lemma~\ref{l:thebigone} that removing a triangle $T'(a, u)$ creates exactly one non-simplicial facet of $Q(n)$, which we will denote by $F(a, u)$.  That is, $F(a, u)$ is the combinatorial bipyramid obtained from the two tetrahedra
\[
	\begin{array}{lll}
		T'(a, u) \cup \{a - u\}, \\
		T'(a, u) \cup \{a + u + 1\},
	\end{array}
\]
of $P'(n)$, by removing their intersection $T'(a, u)$.  The $5$ vertices of $F(a, u)$ are
\[
\V(F(a, u)) = T'(a, u) \cup E(a, u) = \{q(a)\} \cup I(a, u, 1).
\]

We must show that $Q(n)$ is in fact a \emph{polyhedral complex}---that is, that every two faces of $Q(n)$ intersect in a single face. As $P'(n)$ is a simplicial complex, we only need to show that each $F(a, u)$ intersects the other faces of $Q(n)$ properly.
 This follows from the fact that
\[F(a, u) \cap \partial B'(a) = D(a, u),\]
together with Lemma~\ref{l:thebigone}.

Clearly the polyhedral complex $Q(n)$ is homeomorphic to $P'(n)$, hence is a $3$-sphere.
Finally, $Q(n)$ is obtained from $P(n)$ by adding $|A(n)|$ vertices.  Therefore $Q(n)$ has
\[4n+4 + |A(n)| = 5n + 4\]
vertices.  Furthermore, as noted above, we have exactly one non-simplicial facet $F(a, u)$ of $Q(n)$ for each $(a, u) \in A(n) \times [n]$.  That is, $Q(n)$ has
\[|A(n) \times [n]| = n^2\]
non-simplicial facets.
\end{proof}

\textbf{Acknowledgments}:
We thank G\"{u}nter Ziegler for helpful comments on an earlier version of this paper.

\bibliography{gbiblio}{}

\begin{thebibliography}{1}

\bibitem{Alon:fewPolytopes}
Noga Alon.
\newblock The number of polytopes, configurations and real matroids.
\newblock {\em Mathematika}, 33(1):62--71, 1986.

\bibitem{Erickson}
Jeff Erickson.
\newblock New lower bounds for convex hull problems in odd dimensions.
\newblock {\em SIAM J. Comput.}, 28(4):1198--1214 (electronic), 1999.

\bibitem{GoodmanPollack:fewPolytopes-86}
Jacob~E. Goodman and Richard Pollack.
\newblock Upper bounds for configurations and polytopes in {${\mathbf{R}}\sp
  d$}.
\newblock {\em Discrete Comput. Geom.}, 1(3):219--227, 1986.

\bibitem{Grunbaum}
Branko Gr{\"u}nbaum.
\newblock {\em Convex polytopes}, volume 221 of {\em Graduate Texts in
  Mathematics}.
\newblock Springer-Verlag, New York, second edition, 2003.
\newblock Prepared and with a preface by Volker Kaibel, Victor Klee and
  G\"unter M.\ Ziegler.

\bibitem{Kalai-manyspheres}
Gil Kalai.
\newblock Many triangulated spheres.
\newblock {\em Discrete Comput. Geom.}, 3(1):1--14, 1988.

\bibitem{Munkres}
James~R. Munkres.
\newblock {\em Elements of algebraic topology}.
\newblock Addison-Wesley Publishing Company, Menlo Park, CA, 1984.

\bibitem{Pfeifle-Ziegler:many3-spheres}
Julian Pfeifle and G{\"u}nter~M. Ziegler.
\newblock Many triangulated 3-spheres.
\newblock {\em Math. Ann.}, 330(4):829--837, 2004.

\bibitem{Stanley:CohenMacaulayUBC-75}
Richard~P. Stanley.
\newblock The upper bound conjecture and {C}ohen-{M}acaulay rings.
\newblock {\em Studies in Appl. Math.}, 54(2):135--142, 1975.

\bibitem{Ziegler}
G{\"u}nter~M. Ziegler.
\newblock {\em Lectures on polytopes}, volume 152 of {\em Graduate Texts in
  Mathematics}.
\newblock Springer-Verlag, New York, 1995.

\end{thebibliography}
\bibliographystyle{plain}
\end{document}